\documentclass{amsart}
\usepackage[centertags]{amsmath}
\usepackage{graphicx}
\usepackage{dsfont}
\usepackage{amssymb}
\usepackage{enumerate}
\usepackage[table]{xcolor}
\newtheorem{proposition}{Proposition}
\newtheorem{example}[proposition]{Example}
\newtheorem{lemma}[proposition]{Lemma}
\newtheorem{corollary}[proposition]{Corollary}
\newtheorem{theorem}[proposition]{Theorem}

\newcommand{\II}{\mathds{I}}
\newcommand{\DD}{\mathcal{D}}
\begin{document}

\title[Imprecise copulas]{Final solution to the problem of relating a true copula to an imprecise copula}%
\author{Matja\v{z} Omladi\v{c}}%
\address{Institute of Mathematics, Physics and Mechanics, Ljubljana, Slovenia}%
\email{Matjaz@Omladic.net}%
\author{Nik Stopar}%
\address{Faculty of electrical engineering, University of Ljubljana, and Institute of Mathematics, Physics and Mechanics, Ljubljana, Slovenia}%
\email{Nik.Stopar@fe.uni-lj.si}%

\thanks{The authors acknowledge financial support from the Slovenian Research Agency (research core funding No. P1-0222).}
\subjclass{AMS}%

%\date{}%
%\dedicatory{}%
%\commby{}%
% ----------------------------------------------------------------
\begin{abstract}
  In this paper we solve in the negative the problem proposed in this journal (I.\ Montes et al., \textsl{Sklar's theorem in an imprecise setting}, Fuzzy Sets and Systems, \textbf{278} (2015), 48--66) whether an order interval defined by an imprecise copula contains a copula. Namely, if $\mathcal{C}$ is a nonempty set of copulas, then $\underline{C} = \inf\{C\}_{C\in\mathcal{C}}$ and $\overline{C}= \sup\{C\}_{C\in\mathcal{C}}$ are quasi-copulas and the pair $(\underline{C},\overline{C})$ is an imprecise copula according to the definition introduced in the cited paper, following the ideas of $p$-boxes. We show that there is an imprecise copula $(A,B)$ in this sense such that there is no copula $C$ whatsoever satisfying $A \leqslant C\leqslant B$. So, it is questionable whether the proposed definition of the imprecise copula is in accordance with the intentions of the initiators. Our methods may be of independent interest: We upgrade the ideas of Dibala et al. (\textsl{Defects and transformations of quasi-copulas}, Kybernetika, \textbf{52} (2016), 848--865) where possibly negative volumes of quasi-copulas as defects from being copulas were studied.
\end{abstract}
\maketitle
% ----------------------------------------------------------------

\section{Introduction}\label{sec:intro}

Dependence concepts play a crucial role in multivariate statistical literature since it was recognized that the independence assumption cannot describe conveniently the behavior of a stochastic system. One of the main tools in modeling these concepts have eventually become copulas due to their theoretical omnipotence emerging from \cite{Skla} (see also the monographs \cite{AlFrSc,DuSe,Nels}). Namely, they are used to represent and construct joint distribution functions of random vectors in terms of the related one-dimensional marginal distribution functions. As a more general concept, quasi-copulas were introduced in \cite{AlNeSc} and an equivalent definition was given later in \cite{GeQuMoRoLaSe}. Quasi-copulas have interesting applications in several areas, such as fuzzy logic \cite{HaMe,SaTuMe}, fuzzy preference modeling \cite{DeScDeMeDeBa,DiMoDeBa} or similarity measures \cite{DeBaJaDeMe}. Other deep results concerning quasi-copulas can be found in \cite{DeBa,JaDeBaDeMe,NeQuMoRoLaUbFl}.

While copulas are characterized (in the bivariate case) by the nonnegativity of the volume of each subrectangle of the unit square $\II^2$ (where $\II=[0,1]$) which is a Cartesian product of two subintervals of $\II$, this is no longer true for quasi-copulas. This defect of quasi-copulas can be described in several ways, indicating how far away they are from copulas. The authors of \cite{DiSaPlMeKl} introduce several such descriptions and apply them to transform the original quasi-copulas. Note that the sequence of iterative transformations always converges to a copula. This allows them to introduce an equivalence relation on the set of quasi-copulas by grouping quasi-copulas converging to the same copula into an equivalence class. They also give an interesting application of their approach to the so-called imprecise copulas \cite{MoMiMo,MoMiPeVi,PeViMoMi1,PeViMoMi2,Wall}. However, they do not answer one of the main questions proposed there.

Let us recall Definition 1.7.1 in \cite{DuSe}: for quasi-copulas $A$ and $B$ we write $A\leqslant B$ in the \emph{pointwise order} whenever $A(\mathbf{u})\leqslant B(\mathbf{u})$ for all $\mathbf{u}\in\II^2$ (in \cite[Definition 2.8.1]{Nels} this order is denoted by $A\prec B$ and called the \emph{concordance ordering}). If $\mathcal{C}$ is any nonempty set of quasi-copulas, then $\underline{C}= \inf\{C\}_{C\in\mathcal{C}}$ and $\overline{C}= \sup\{C\}_{C\in\mathcal{C}}$ are quasi-copulas by \cite[Theorem 6.2.5]{Nels}. In this respect the set of quasi-copulas is a complete lattice and may be actually seen as an order completion of the set of all copulas. Following the ideas of $p$-boxes one may consider in the case that $\mathcal{C}$ contains copulas only, the pair $(\underline{C},\overline{C})$ as an ``imprecise copula'' representing the set of copulas $\mathcal{C}$. While the authors of \cite{MoMiPeVi,PeViMoMi1,PeViMoMi2} introduce their definition of an imprecise copula so that this is true, they propose a question in the other direction whether every imprecise copula can be obtained in this way. The main purpose of this paper is to answer this 6 years old question in the negative. Actually, we prove much more: There is an imprecise copula $(A,B)$ in their sense such that there is no copula $C$ whatsoever such that $A \leqslant C\leqslant B$ (i.e.\ the order interval generated by $(A,B)$ contains $C$). In order to do that we need to use and extend substantially the methods of \cite{DiSaPlMeKl}.

The comment immediately following Definition 6 in \cite{PeViMoMi1} says: ``We are using the terminology imprecise copula in the definition above because we intend it as a mathematical model for the imprecise knowledge of a copula.'' This might be questionable now that we know that the order interval of quasi-copulas defined by an imprecise copula may contain no copula at all.

The paper is organized as follows. Section 2 gives a novel approach to discrete copulas and quasi-copulas, an important tool we introduce to study the proposed problem. In Section 3 we present and slightly extend the methods of \cite{DiSaPlMeKl}, and give an example in the (discrete) quasi-copula setting indicating that the desired counterexample cannot be made using these methods only. (We believe that analogous example can be done in a general setting but omit the lengthy calculations and rather give later a more powerful example in details.) The methods of this section are substantially upgraded in Section 4, where we present (among other things) a necessary and sufficient condition for an imprecise copula $(A,B)$ that the order interval it generates contains a copula. The desired counterexample is then given in Section 5 together with additional conditions that quasi-copula $A$ equals the infimum of copulas contained in this interval, respectively that quasi-copula $B$ equals the supremum of copulas contained in this interval.

\section{Discrete copulas vs.\ continuous copulas}\label{sec:discrete}

In this section we present discrete copulas in somewhat more general fashion than in \cite{QuMoSe}. These copulas together with their interplay with the usual copulas will play a crucial role in developing our main results later.

Actually, we need to start even in a bit more general way and consider real functions $A$ defined either on the unit square $\mathds{I}^2$, where $\mathds{I}=[0,1]$, or on a mesh within this square $\Delta=\delta_x\times \delta_y$ determined by some points
\[
    \delta_x=\{0=x_0<x_1<\cdots<x_p=1\}\quad\mbox{and}\quad \delta_y=\{0=y_0<y_1<\cdots<y_q=1\}.
\]
Consider the rectangles whose corners are the intersections of verticals going through two consecutive $x_i$'s with horizontals going through two consecutive $y_j$'s. We will say that the mesh $\Delta$ is \emph{determined by these rectangles}. So, we will either have $A:\II^2\rightarrow \mathds{R}$ or $A:\Delta \rightarrow \mathds{R}$ and call the first case the \textsl{general case} and the second one the \textsl{discrete case}. For the function $A$ we will often assume  (1) that it is 1-increasing, i.e.\ that it is nondecreasing in each of the two variables, (2) that it is continuous in the general case. Sometimes we will want to unify the two cases and write $\DD$ to either mean $\Delta$ or $\II^2$.

Let us recall the definition of the bilinear interpolation on a rectangle and give the proof of its uniqueness for the sake of completeness. Let $R\subseteq\mathds{I}^2$ be a positively oriented rectangle defined by corners $\mathbf{a}, \mathbf{b}, \mathbf{c},$ and $\mathbf{d}$, with $\mathbf{a}$ as the southwest corner. This will be assumed as a standard notation of corners of a rectangle unless specified otherwise. Assume we know the values of a function $A$ at the corners of $R$.

\begin{proposition}\label{bilinear}
  Given the values of a 1-increasing function $A$ at the corners of a rectangle $R$ there exists a unique function $A$ on $R$ such that
\begin{description}
  \item[(a)] its values coincide with the starting values at the corners;
  \item[(b)] each one-dimensional section parallel to the axes is linear.
\end{description}
\end{proposition}

\begin{proof}
 If the function is constant, then the proposition is easy. If not, then $A(\mathbf{c})>A(\mathbf{a})$ by the fact that $A$ is 1-increasing. We assume with no loss that $R=\mathds{I}^2$, and that $A(0,0)=0$ (after subtracting a constant from $A$, if necessary) and $A(1,1)=1$ (after multiplying $A$ by a constant, if necessary). Let $A(1,0)= \alpha$ and $A(0,1)=\beta$ to write down the desired solution
\[
    A(x,y)=\alpha x+\beta y +(1-\alpha-\beta)xy.
\]
\end{proof}

The function on $R$ obtained in this way will be called a \textsl{bilinear interpolation} of $A$ through its values at the corners. For any rectangle $R\subseteq\mathds{I}^2$ with standard corners $\mathbf{a}, \mathbf{b}, \mathbf{c},$ and $\mathbf{d}$, and any function $A$ defined at least on these corners we let the volume of $R$ with respect to $A$ be equal to
\[
    V_A(R)= A(\mathbf{a})+A(\mathbf{c})-A(\mathbf{b})-A(\mathbf{d}).
\]

\begin{corollary}\label{positive}
  The bilinear interpolation $A$ of a 1-increasing function defined on the corners of $R$ is 1-increasing. Moreover, for every subrectangle $R_1\subseteq R$ we have:
\begin{description}
  \item[(a)] $V_A(R_1)>0$ if and only if $V_A(R)>0$
  \item[(b)] $V_A(R_1)<0$ if and only if $V_A(R)<0$
  \item[(c)] $V_A(R_1)=0$ if and only if $V_A(R)=0$
\end{description}
\end{corollary}

\begin{proof}
  Again we assume with no loss that $R=\mathds{I}^2$ and that the values of $A$ on the corners are as assumed in the proof of Proposition \ref{bilinear}. Recall the formula from that proof to see that the first partial derivatives of $A$ are respectively equal to $\alpha+ (1-\alpha-\beta)y$ (which is no smaller than zero if and only if $A$ is increasing in the direction of $x$) and $\beta+ (1-\alpha-\beta)x$ (which is no smaller than zero if and only if $A$ is increasing in the direction of $y$). The second mixed partial derivative of $A$ equals $1-\alpha-\beta$ which is exactly $V_A(R)$. So, the 2-volume of any rectangle $R_1\subseteq R$ with sides of length $\gamma$ and $\delta$ equals $(1-\alpha-\beta)\gamma\delta$ and the corollary follows.
\end{proof}

From now on all our functions will be 1-increasing. If $A$ is defined on $\mathds{I}^2$ we get a function defined on a mesh $\Delta= \delta_x \times\delta_y$ simply by taking the restriction $A|_\Delta$. On the other hand, if $A$ is defined on the mesh $\Delta$ we can extend it to a function on $\mathds{I}^2$ by taking the bilinear interpolation of $A$ on each of the rectangles determined by the mesh; we will denote this interpolation by $A^{\mathrm{BL}}$. Observe that this definition depends also on the mesh, but we are not pointing that out in this notation when the mesh is understood.

\begin{proposition}
  If $A$ is defined on the mesh $\Delta=\delta_x\times\delta_y$, then $A^{\mathrm{BL}}$ is an absolutely continuous function on $\mathds{I}^2$ and $A^{\mathrm{BL}}|_\Delta=A$.
\end{proposition}

%From now on $\Delta$ will be either the unit square $\Delta= \mathds{I}^2$, or a mesh within this square $\Delta= \delta_x \times \delta_y$.
For a function $A$ (whether general or discrete) we will say that
\begin{description}
  \item[(A)] $A$ is \emph{grounded} if $A(x,0)=0$ and $A(0,y)=0$ for all $x,y$ such that $(x,0),(0,y)\in\DD$;
  \item[(B)] number 1 is its \emph{neutral element} if $A(x,1)=x$ and $A(1,y)=y$ for all $x,y$ such that $(x,1),(1,y)\in \DD$;
  \item[(C)] $A$ is \emph{$2$-increasing} if the volume of every rectangle with corners in $\DD$ is nonnegative;
  \item[(D)] $A$ is quasi-$2$-increasing if this condition is fulfilled for all rectangles with corners in $\DD$ that have non-empty intersection with the sides of the unit square $\mathds{I}^2$.
\end{description}

Now, if $\DD=\mathds{I}^2$ and $A$ is grounded, has 1 as a neutral element and is 2-increasing, then $A$ is called a \emph{copula}. It is called a \emph{quasi-copula} if it is grounded, has 1 as a neutral element and is quasi-2-increasing. If $\DD= \Delta$ then in the two respective cases $A$ is called a \emph{discrete copula}, respectively a \emph{discrete quasi-copula}. Observe that our notion of discrete copula is slightly more general than the one introduced in \cite[Section 3.1.1]{DuSe} and our notion of discrete quasi-copula is slightly more general than the one introduced in \cite{QuMoSe}. Recall that in the definition of a quasi-copula condition \textbf{(D)} can be equivalently replaced by (cf.\ \cite[p.\ 236]{Nels})
\begin{description}
  \item[(D')] $A$ is increasing and $1$-Lipschitz in each variable.
\end{description}
We shall not study this approach in details.

\begin{proposition}\label{copula_properties}
  Let $A$ be a function defined on $\Delta$, then
\begin{description}
  \item[(a)] $A$ is grounded if and only if $A^{\mathrm{BL}}$ is grounded;
  \item[(b)] 1 is a neutral element for $A$ if and only if 1 is a neutral element for $A^{\mathrm{BL}}$;
  \item[(c)] $A$ is 2-increasing if and only if $A^{\mathrm{BL}}$ is 2-increasing;
  \item[(d)] $A$ is quasi-2-increasing if and only if $A^{\mathrm{BL}}$ is quasi-2-increasing;
  \item[(e)] $A$ is a discrete copula if and only if $A^{\mathrm{BL}}$ is a copula;
  \item[(f)] $A$ is a discrete quasi-copula if and only if $A^{\mathrm{BL}}$ is a quasi-copula;
\end{description}
\end{proposition}

\begin{proof}
  \textbf{(a)} and \textbf{(b)} are trivial consequences of the definition. \textbf{(c):} If $A$ is 2-increasing, then, in particular, the rectangles that belong to the mesh all have nonnegative $A^{\mathrm{BL}}$-volume. However, every rectangle is a union of rectangles that are subrectangles of those that belong to the mesh and we are done by Corollary~\ref{positive}. The opposite implication is clear. Now, one implication of part \textbf{(d)} is clear again. For the proof in the other direction we choose an arbitrary rectangle $R\subseteq\II^2$ whose intersection with one of the sides of $\II^2$ is nonempty. We will only treat the case when the intersection with the south side is nonempty since the other four cases go similarly. Denote by $\mathbf{a}$, $\mathbf{b}$, $\mathbf{c}$, and $\mathbf{d}$, the defining corners of $R$, by $x_1$ the first coordinate of $\mathbf{a}$ respectively $\mathbf{d}$, by $x_2$ the first coordinate of $\mathbf{b}$ respectively $\mathbf{c}$, and by $y$ the second coordinate of $\mathbf{c}$ respectively $\mathbf{d}$, while the second coordinate of $\mathbf{a}$ respectively $\mathbf{b}$ is zero. By definition $V_{A^{\mathrm{BL}}}(R)$ is a linear function in each of the variables $x_1,x_2$, $y$, the first two coming from two respective intervals each made of two consecutive members of the mesh $\delta_x$, the third one coming in an analogous way from $\delta_y$. Now, a multilinear function can have at most one zero on any segment parallel to any of the edges. By assumption it is nonnegative at the corners, consequently on any edge, therefore on any face and finally anywhere within this 3-dimensional rectangular cuboid. Parts \textbf{(e)} and \textbf{(f)} follow easily from the previous assertions.
\end{proof}

\textbf{Remark.} Note that the ``easier'' part, i.e.\ the ``if'' part, of all these claims remains true when $A^{\mathrm{BL}}$ is replaced by an arbitrary real valued function, say $\breve{A}$, on $\II^2$ such that $\breve{A}|_\Delta=A$.\\

Observe in passing that our assumption of functions being 1-increasing in order to be able to make bilinear interpolation of the discrete functions defined on the mesh under consideration does not narrow down the applications of our results since most of our functions will be quasi-copulas which are always 1-increasing by \cite[Theorem 7.2.1]{DuSe}.

Following \cite{MoMiPeVi} (cf.\ also \cite{DiSaPlMeKl,OmSk}) we call a pair $(A, B)$ of functions on $\DD$ an \emph{imprecise copula} if \textbf{(A)} they are grounded, \textbf{(B)} each of them has 1 as a neutral element, and
\begin{description}
  \item[(IC1)]  $A(\mathbf{a})+B(\mathbf{c})-A(\mathbf{b})- A(\mathbf{d})\geqslant0$;
  \item[(IC2)] $B(\mathbf{a})+A(\mathbf{c})-A(\mathbf{b})- A(\mathbf{d})\geqslant0$;
  \item[(IC3)] $B(\mathbf{a})+B(\mathbf{c})-B(\mathbf{b})- A(\mathbf{d})\geqslant0$;
  \item[(IC4)] $B(\mathbf{a})+B(\mathbf{c})-A(\mathbf{b})- B(\mathbf{d})\geqslant0$
\end{description}
for each rectangle $R\subseteq\DD$ defined by corners $\mathbf{a}, \mathbf{b}, \mathbf{c},$ and $\mathbf{d}$ in the standard way. In the general case it is known \cite{MoMiPeVi,DiSaPlMeKl,OmSk} and not difficult to verify that, for each imprecise copula $(A,B)$, we have that $A$ and $B$ are quasi-copulas and $A\leqslant B$.

Observe that formally the same verification that yields this result for the general case, gives also the analogous result for the discrete case. So, complying with the above approach to discrete versions of copulas and quasi-copulas we introduce a \textsl{discrete imprecise copula} as a pair of functions $(A,B)$ defined on a mesh $\Delta$ satisfying axioms \textbf{(A)}, \textbf{(B)}, \textbf{(IC1)}, \textbf{(IC2)}, \textbf{(IC3)}, and \textbf{(IC4)}. As the first justification for introducing this notion we give

\begin{proposition}
  If a pair $(A,B)$ defined on a mesh $\Delta$ is a discrete imprecise copula, then the pair $(A^{\mathrm{BL}}, B^{\mathrm{BL}})$ is an imprecise copula. Also, if $(A,B)$ is an imprecise copula, then $(A|_\Delta,B|_\Delta)$ is a discrete imprecise copula.
\end{proposition}

\begin{proof}
  Assume $(A,B)$ is a discrete imprecise copula on a mesh, let $R$ be any rectangle in $\II^2$ with corners denoted standardly by $\mathbf{a}, \mathbf{b}, \mathbf{c},$ and $\mathbf{d}$, and let $C(R)$ be the lefthand side of any of the conditions \textbf{(IC\emph{i})} for $i\in\{1,2,3, 4\}$ in which $A$ is replaced by $A^{\mathrm{BL}}$ and $B$ is replaced by $B^{\mathrm{BL}}$. By assumption $C(R)\geqslant0$ if the four corners are determined by the mesh. Now, assume $\mathbf{b}'$ is another corner of the mesh, horizontally adjacent to $\mathbf{b}$, and $\mathbf{c}'$ is horizontally adjacent to $\mathbf{c}$. Define $\mathbf{b}_t=t\mathbf{b}'+ (1-t)\mathbf{b}$, and $\mathbf{c}_t= \mathbf{c}'+(1-t) \mathbf{c}$, and let $R_t$ be the rectangle with corners $\mathbf{a}, \mathbf{b}_t, \mathbf{c}_t,$ and $\mathbf{d}$ for $t\in\II$. When $t$ moves from 0 to 1, $R_t$ changes from $R$ to the rectangle obtained from $R$ by shifting the east side one grid further. Since $C(R_t)$ is nonnegative at the endpoints $t=0$ and $t=1$ and since it is a linear function of $t$ between the two points, it has to be nonnegative throughout since a linear function cannot have more than one zero on a segment. Using a similar argument one can conclude that the point $\mathbf{c}_t$ can be moved vertically northwards (and $\mathbf{d}$ simultaneously in the same way) without changing the sign of $C(R_t)$. Once we establish that  $\mathbf{c}$ can be assumed arbitrary, we do the same for $\mathbf{a}$. The other implication is clear.
\end{proof}

\section{Bounds of an imprecise copula determined by defects}\label{sec:bounds}

It is our aim in this section to apply the theory presented in Dibala et al.\ \cite{DiSaPlMeKl} on some examples of imprecise copulas. In the next section we will develop this theory further and it will become an important ingredient on the way to our main results. In \cite{DiSaPlMeKl} the authors introduce defects of quasi-copulas and give an extensive study of this notion which helps understanding how far a quasi-copula is from a copula. We start this section by a brief summary of symbols, definitions and some key results of that paper to be needed in the sequel. This way we want to assist an interested reader in following our considerations.

The authors of \cite{DiSaPlMeKl} first introduce four mappings associating each point of the unit square $\II^2$ to a set of rectangles contained in $\II^2$. Let $\mathbf{R}$ mean the set of all rectangles $R\subseteq \II^2$ and denote for any rectangle $R$ its standard corners by $\mathbf{a}_R, \mathbf{b}_R, \mathbf{c}_R,$ and $\mathbf{d}_R$. Next, for any $\mathbf{x}\in\II^2$ we let
\begin{equation}\label{sets_rect}
\begin{split}
   \mathcal{R}_\nearrow(\mathbf{x})=\{R\in\mathbf{R};\,\mbox{s.t.}\ \mathbf{x} =\mathbf{a}_R\},\ \ & \mathcal{R}_\swarrow(\mathbf{x}) =\{R\in\mathbf{R};\,\mbox{s.t.}\ \mathbf{x}= \mathbf{c}_R\}, \\
   \mathcal{R}_\searrow(\mathbf{x}) =\{R\in\mathbf{R};\,\mbox{s.t.}\ \mathbf{x}= \mathbf{d}_R\},\ \  & \mathcal{R}_\nwarrow(\mathbf{x}) =\{R\in\mathbf{R};\,\mbox{s.t.}\ \mathbf{x}=\mathbf{b}_R\}.
\end{split}
\end{equation}
Furthermore, for any function $C:\DD \rightarrow \mathds{R}$, where $\DD= \mathds{I}^2$, the full unit square, they introduce six measures that show how far a function (if quasi-copula) is from a copula
\[
    \begin{split}
   D_\nearrow^C(\mathbf{x})=\inf_{R\in\mathcal{R}_\nearrow(\mathbf{x})}V_C(R),\ \ & D_\swarrow^C(\mathbf{x})=\inf_{R\in\mathcal{R}_\swarrow(\mathbf{x})}V_C(R), \\
   D_\searrow^C(\mathbf{x})=\inf_{R\in\mathcal{R}_\searrow(\mathbf{x})}V_C(R),\ \   & D_\nwarrow^C(\mathbf{x})=\inf_{R\in\mathcal{R}_\nwarrow(\mathbf{x})}V_C(R);
    \end{split}
\]
as well as
\[
    D_M^C=D_\nearrow^C\land D_\swarrow^C\ \ \mbox{and}\ \ D_O^C=D_\searrow^C\land D_\nwarrow^C.
\]
These measures are called suggestively in the above order the \emph{northeast, the southwest, the southeast, the northwest, the main, and the opposite defect}. It is easy to see that in these definitions and in considerations to follow one can replace $\DD=\mathds{I}^2$, the full unit square, by $\DD=\Delta=\delta_x\times \delta_y$, a fixed finite mesh within it. It turns out that (\cite[Proposition 3.2]{DiSaPlMeKl}) \emph{in case that $C$ is a quasi-copula, then it is a copula if and only if one (and, subsequently each) of the defect functions $D_\nearrow^C, D_\swarrow^C, D_\searrow^C,$ and $D_\nwarrow^C$ is identically zero.}

For a quasi-copula $C$ they also introduce
\begin{alignat*}{3}
C_\nearrow  & = C-D_\nearrow^C  &\qquad  C_\searrow  & = C+D_\searrow^C  &\qquad  C_M  & = C-D_M^C\\ C_\swarrow  & = C-D_\swarrow^C  &  C_\nwarrow  & = C+D_\nwarrow^C &  C_O  & = C+D_O^C
\end{alignat*}
They show, among other things, that (\cite[Theorem 4.3]{DiSaPlMeKl}) \emph{in case that $C$ is a quasi-copula each of the six functions $C_\nearrow, C_\swarrow, C_\searrow, C_\nwarrow, C_M$, and $C_O$ is a quasi-copula.} For a quasi-copula $C$ they observe immediately after that theorem that
\[
    C_O\leqslant C\leqslant C_M.
\]

So, to any imprecise copula $(A,B)$ one can introduce quasi-copulas $A_M$ and $B_O$ such that $A\leqslant A_M$ and $B_O\leqslant B$. Actually in \cite[Theorem 5.2]{DiSaPlMeKl} they prove that \emph{a pair $(A,B)$ of quasi-copulas is an imprecise copula if and only if $B\geqslant A_M$ and $B_O\geqslant A$.} Their theory extends easily to the discrete case introduced in Section \ref{sec:discrete}. We will expand their tools further and prove somewhat more. However, let us start with an auxiliary result.

For any function $C:\DD \rightarrow \mathds{R}$, where either $\DD=\mathds{I}^2$, the full unit square, or $\DD=\Delta=\delta_x\times \delta_y$, a fixed finite mesh within it, we can reflect one of the variables by sending either $x\mapsto 1-x$ or $y\mapsto 1-y$. It is easy to compute the resulting function, i.e.\  adjust this operation so that the function stays grounded and has 1 as a neutral element also after the reflection (cf.\ \cite[Theorem 2.4.4]{Nels} and also \cite[Section 1.7.3]{DuSe}). If $C$ is a quasi-copula one can perform the same adjustment; in both cases we denote the result by $C^\sigma$. We do the same in the discrete case, where we need to adjust the mesh according to the transformation, if necessary. It turns out that in all the four cases we have formally either
\[
    C^\sigma(x,y)=y-C(1-x,y),\ \ \mbox{or}\ \ C^\sigma(x,y)=x-C(x,1-y).
\]

\begin{lemma}\label{reflection} Let $C$ be an arbitrary quasi-copula, then every reflection $\sigma$ exchanges the main and the opposite role of every corner of any rectangle and consequently:
  \begin{description}
    \item[(a)] $(C^\sigma)_M=(C_O)^\sigma$
    \item[(b)] $(C^\sigma)_O=(C_M)^\sigma$
  \end{description}
\end{lemma}

\begin{proof}
  Transformation $\sigma: x\mapsto 1-x$ sends a rectangle $R$ with standard corners $\mathbf{a}_R, \mathbf{b}_R, \mathbf{c}_R,$ and $\mathbf{d}_R$ into the rectangle $\sigma(R)$ with standard corners $\mathbf{a}_{\sigma(R)}= \sigma(\mathbf{b}_R), \mathbf{b}_{\sigma(R)}= \sigma(\mathbf{a}_R), \mathbf{c}_{\sigma(R)}=\sigma(\mathbf{d}_R)$, and $\mathbf{d}_{\sigma(R)}=\sigma( \mathbf{c}_R)$. So, when a point $\mathbf{x}$ is sent to $\sigma(\mathbf{x})$, the corresponding sets \eqref{sets_rect} of rectangles exchange their role. In particular, the rectangles of $\mathcal{R}_\nearrow$ and $\mathcal{R}_\nwarrow$ exchange and similarly the rectangles of $\mathcal{R}_\searrow$ and $\mathcal{R}_\swarrow$ do. We compute the defects before and after the transformation and compare the results. After going through the above definition of defects we come up with the fact that $D_M^C$ and $D_O^C$ are exchanged. For the other reflection $\sigma: y\mapsto 1-y$ the particularities of the proof are analogous, but the conclusion is the same.
\end{proof}

\begin{theorem}\label{interval}
  For arbitrary quasi-copulas $A$ and $B$ the pairs $(A,A_M)$ and $(B_O,B)$ are imprecise copulas.
\end{theorem}

\begin{proof}
The proof of the desired fact for the pair $(A,A_M)$ follows from the proof of the analogous fact for the pair $(B_O,B)$ by Lemma \ref{reflection}. So, it suffices to show it for the latter case. Choose any quasi-copula $B$ and observe that by \cite[Theorem 5.2]{DiSaPlMeKl} the pair $(B_O,B)$ will be an imprecise copula as soon as we prove that $(B_O)_M\leqslant B$. After denoting $Q=B_O$ we use considerations of \cite{DiSaPlMeKl} to get
\[
    Q_M=Q-D_M^Q=Q-D^Q_\nearrow \land D^Q_\swarrow=(Q-D^Q_\nearrow)\lor (Q-D^Q_\swarrow)=Q_\nearrow\lor Q_\swarrow.
\]
So, in order to show that $Q_M\leqslant B$ we need to show that both $Q_\nearrow \leqslant B$ and $Q_\swarrow\leqslant B$. Observe that by the definition of the mapping $Q\mapsto Q_\nearrow$ and of $Q=B_O$ we get
\[
    Q_\nearrow=Q-D_\nearrow^Q=B+D_O^B-D_\nearrow^{B_O}.
\]
In order to show that $Q_\nearrow\leqslant B$, it suffices to prove that
\begin{equation}\label{eq:defects}
  D_O^B(\mathbf{x})\leqslant D_\nearrow^{B_O}(\mathbf{x}), \quad\mbox{for all}\quad \mathbf{x}\in[0,1]^2.
\end{equation}
For a fixed point $\mathbf{x}\in[0,1]^2$ we denote following the notation of \cite[Section 3]{DiSaPlMeKl} by $\mathcal{R}_\nearrow (\mathbf{x})$ the set of rectangles (contained in $\II^2$ of course) with the southwest corner equal to $\textbf{x}$. The right-hand side of Inequality \eqref{eq:defects} equals the infimum of the volumes with respect to quasi-copula $B_O$ of all rectangles $\mathcal{R}_\nearrow(\mathbf{x})$. Now, we choose a rectangle $R$ whose corners are denoted by $\mathbf{a}=\mathbf{x}, \mathbf{b}, \mathbf{c}, \mathbf{d}$ in the standard way. Using the definition of $B_O$ and of the volume we compute
\begin{equation}\label{eq:volume}
   D_O^B(\mathbf{x})-V_{B_O}(R)=D_O^B(\mathbf{x})-V_B(R)- V_{D_O^B} (R) =-V_B(R)-D_O^B(\mathbf{c})+ D_O^B(\mathbf{b})+D_O^B(\mathbf{d}).
\end{equation}
We want to show that the right-hand side of \eqref{eq:volume} is non-positive. By the remark following Definition 3.1 in \cite{DiSaPlMeKl} we know that defect $D_O^B(\mathbf{c})$ is attained at some rectangle $P\in \mathcal{R}_\nwarrow(\mathbf{c})\cup\mathcal{R}_\searrow(\mathbf{c})$ so  that $D_O^B(\mathbf{c})=V_B(P)$. We will consider only the case that $P\in \mathcal{R}_\nwarrow(\mathbf{c})$ since it turns out that the other one goes in a similar way. We denote by $\textbf{y}$ the southwest corner of $P$ and consider two cases. Assume first that $\textbf{y}$ lies on the line segment connecting $\textbf{d}$ and $\textbf{c}$. Divide the rectangle $R$ into two rectangles $R_1$ and $R_2$ so that $\textbf{y}$ is the northeast corner of $R_1$ and therefore the northwest corner of $R_2$. Note that
\[
    V_B(R)+D_O^B(\mathbf{c})=V_B(R)+V_B(P)=V_B(R_1)+V_B(R_2\cup P)\geqslant D_O^B(\mathbf{d})+D_O^B(\mathbf{b})
\]
implying that the right-hand side of \eqref{eq:volume} is non-positive in this case. Now, if $\textbf{y}$ does not lie on the line segment connecting $\textbf{d}$ and $\textbf{c}$, then $\textbf{d}$ lies on the line segment connecting $\textbf{y}$ and $\textbf{c}$, so we can divide the rectangle $P$ into two rectangles $P_1$ and $P_2$ so that $\textbf{d}$ is the southeast corner of $P_1$ and therefore the southwest corner of $P_2$. Consequently, we have in this case
\[
    V_B(R)+D_O^B(\mathbf{c})=V_B(P_1)+V_B(P_2\cup R)\geqslant D_O^B(\mathbf{d})+D_O^B(\mathbf{b})
\]
so that the left-hand side of \eqref{eq:volume} is always non-positive. When taking the infimum over all rectangles $R\in\mathcal{R}_\nearrow (\mathbf{x})$, one gets \eqref{eq:defects} and we are done.
\end{proof}

We will call an imprecise copula $(C,D)$ an imprecise subcopula of the imprecise copula $(A,B)$ if
\[
    A\leqslant C\leqslant D\leqslant B
\]

\begin{corollary}
  For any imprecise copula $(A,B)$ there exist two imprecise subcopulas, one with the same lower bound $(A,A_M)$, one with the same upper bound $(B_O,B)$.
\end{corollary}

We will now concentrate on the question proposed in \cite{MoMiPeVi,PeViMoMi2} and studied in \cite{DiSaPlMeKl}. The general case of the following fact is shown in these papers, while the discrete case follows easily using analogous considerations.

\begin{lemma}\label{lem:interval}
  Let $(C_i)_{i\in I}$ be a family of copulas, either general or discrete. Then the functions $\underline{C},\overline{C}:\DD \rightarrow\II$ defined by
\[
    \underline{C}=\bigwedge_{i\in I}C_i\ \ \mbox{and}\ \ \overline{C}=\bigvee_{i\in I}C_i
\]
give rise to an imprecise copula $(\underline{C},\overline{C})$.
\end{lemma}

So, there are imprecise copulas $(A,B)$ containing copulas in the ordered interval from $A$ to $B$, i.e.\
\[
[A,B]=\{C\,;\,A\leqslant C\leqslant B\}.
\]
The question is whether all imprecise copulas can be obtained in the way given in the lemma above. Even more, it is not known whether for every imprecise copula $(A,B)$ there is a copula $C\in[A,B]$. Theorem \ref{interval} suggests a method to get the desired copula $C$ provided that it exists. Start with an arbitrary imprecise copula $(A,B)$, general or discrete. Then (1) Define $B'=A_M$ and observe that $A \leqslant B'\leqslant B$ and that $(A,B')$ is an imprecise copula by Theorem \ref{interval}. (2) Define $A'=B'_O$ and observe that $A \leqslant A'\leqslant B'\leqslant B$ and that $(A',B')$ is an imprecise copula by Theorem \ref{interval}. Using this procedure repeatedly we get a sequence of imprecise copulas $(A^n,B^n)$, where
\begin{equation}\label{eq:recursion}
  A^{n+1}=(B^{n+1})_O\ \ \mbox{and}\ \ B^{n+1}=(A^n)_M.
\end{equation}
%in the order sense contained in each other. Here we started with $A^0=A$ and $B^0=B$, where $(A,B)$ is the starting imprecise copula given above.

\begin{proposition}\label{exammple_general} Let $(A,B)$ be an imprecise copula and define sequences $A^n$ and $B^n$ by $A^0=A$, $B^0=B$, and Equation \eqref{eq:recursion} above. Then we have
  \begin{description}
    \item[(a)] The sequence $A^n$ is increasing, the sequence $B^n$ is decreasing and they are uniformly converging. Denote their respective limits by $\breve{A}$ and $\breve{B}$.
    \item[(b)] Each pair $(A^n,B^n)$ is an imprecise subcopula of $(A^{n-1},B^{n-1})$.
    \item[(c)] The limiting pair $(\breve{A},\breve{B})$ is also an imprecise subcopula of each $(A^n,B^n)$, and has the property  $(\breve{A})_M=\breve{B}$ and $(\breve{B})_O=\breve{A}$.
  \end{description}
\end{proposition}

\begin{proof}
  Since the sequence of the left-hand sides $A^n$ is a point-wise increasing sequence of quasi-copulas and bounded above (by the Fr\'{e}chet Hoeffding upper bound, say), it converges to a quasi-copula to be denoted by $\breve{A}$. Similarly the sequence of the right-hand sides $B^n$ is point-wise decreasing, bounded below (by the Fr\'{e}chet Hoeffding lower bound, say), so its limit exists, will be denoted by $\breve{B}$ and is a quasi-copula. Furthermore, it is not hard to see that $(\breve{A}, \breve{B})$ is an imprecise subcopula of all imprecise copulas in this sequence which are imprecise subcopulas of each other. A standard consideration shows that the sequences $A^n\to \breve{A}$ and $B^n\to \breve{B}$ converge uniformly. There is a general theorem \cite[Theorem 1.7.6]{DuSe} saying that point-wise convergence yields uniform convergence for copulas. Furthermore, everything that is needed in the proof of this theorem is valid for quasi-copulas as well (cf.\ \cite[Chapter 7]{DuSe}). Now, the fact that %$A^n\to \breve{A}$ implies
\[
    A^n\to \breve{A}\ \ \mbox{uniformly, implies}\ \ D_{\nearrow}^{A^n}\to D_{\nearrow}^{\breve{A}}\ \ \mbox{and}\ \ D_{\swarrow}^{A^n}\to D_{\swarrow}^{\breve{A}}\ \ \mbox{so that}\ \ D_M^{A^n}\to D_M^{\breve{A}}.
\]
Consequently,
\[
    B^{n+1}=(A^n)_M=A^n-D_M^{A^n}\to \breve{A}-D_M^{\breve{A}}= (\breve{A})_M.
\]
We have thus seen that $\breve{B}=(\breve{A})_M$ and we can prove similarly that $\breve{A}=(\breve{B})_O$.
\end{proof}

Now, if $\breve{A}=\breve{B}$, this must be a true copula, so that we have found the desired copula between $A$ and $B$. On the other hand, if  $\breve{A}\ne\breve{B}$, this means that one cannot find a copula between $A$ and $B$ only by means of the methods described in this section. We believe we can find an example of a general copula with this property using Proposition \ref{exammple_general}. However, we omit the lengthy calculations and present a more valuable example in Section \ref{sec:main}, although we may decide to publish it in a forthcoming paper if the interest for that develops.

Here, let us present an example of a discrete imprecise copula with this property. We fix a mesh $\Delta=\delta_x \times \delta_y$ where $\displaystyle x_k=y_k =\frac{k}{7}$ for $k=0,1,\ldots, 7$, and let the value of a discrete imprecise copula $(\langle A\rangle, \langle B\rangle)$ be given by
\[
    \langle A\rangle\left(\frac{j}{7},\frac{k}{7}\right)=[A]_{j+1,k+1}\ \ \mbox{and}\ \ \langle B\rangle \left(\frac{j}{7},\frac{k}{7}\right)= [B]_{j+1,k+1}\ \ \mbox{for}\ \ j,k=0,1,\ldots,7,
\]
where $[A]$ and $[B]$ are matrices
\begin{equation}\label{OM:data}
  [A]= \frac{1}{7}\left[
\begin{array}{cccccccc}
 0 & 0 & 0 & 0 & 0 & 0 & 0 & 0 \\
 0 & 0 & 0 & 0 & 0 & 0 & 1 & 1 \\
 0 & 0 & 0 & 0 & 0 & 1 & 2 & 2 \\
 0 & 0 & 0 & 0 & 1 & 2 & 2 & 3 \\
 0 & 0 & 0 & 1 & 2 & 2 & 3 & 4 \\
 0 & 0 & 1 & 2 & 2 & 3 & 4 & 5 \\
 0 & 1 & 2 & 2 & 3 & 4 & 5 & 6 \\
 0 & 1 & 2 & 3 & 4 & 5 & 6 & 7 \\
\end{array}
\right], \quad
  [B]=\frac{1}{7}\left[
\begin{array}{cccccccc}
 0 & 0 & 0 & 0 & 0 & 0 & 0 & 0 \\
 0 & 0 & 0 & 0 & 0 & 1 & 1 & 1 \\
 0 & 0 & 0 & 0 & 1 & 2 & 2 & 2 \\
 0 & 0 & 0 & 1 & 2 & 2 & 3 & 3 \\
 0 & 0 & 1 & 2 & 2 & 3 & 4 & 4 \\
 0 & 1 & 2 & 2 & 3 & 4 & 4 & 5 \\
 0 & 1 & 2 & 3 & 4 & 4 & 5 & 6 \\
 0 & 1 & 2 & 3 & 4 & 5 & 6 & 7 \\
\end{array}
\right].
\end{equation}
%Moreover, define $A=\langle A\rangle^{\mathrm{BL}}$ and $B=\langle B\rangle^{\mathrm{BL}}$. A simple computation reveals that
%\begin{equation*}%\label{OM}  \langle B\rangle=\langle A\rangle_M\ \ \mbox{and}\ \ \langle A\rangle=\langle B\rangle_O.      \end{equation*}     We perform the procedure described in \eqref{eq:recursion} and observe inductively that $A^n|_\Delta=\langle A\rangle$ and $B^n|_\Delta=\langle B\rangle$ for all $n$. Consequently we get in the limit $\breve{A}|_\Delta=\langle A\rangle$ and $\breve{B}|_\Delta=\langle B\rangle$ so that $\breve{A}\neq\breve{B}$ as desired.

\begin{example}
  There exists a discrete imprecise copula $(\langle A\rangle, \langle B\rangle)$, actually we may consider the one obtained from \eqref{OM:data}, such that $\langle A\rangle\neq\langle B\rangle$, $\langle A\rangle_M=\langle B\rangle$ and $\langle B\rangle_O= \langle A\rangle$.
\end{example}

\section{Existence of a copula inside an imprecise copula}

It is time to develop the main tools needed on the way to our main results. Let $A$ and $B$ be a pair of real valued functions such that $A\leqslant B$. This assumption will later be narrowed down to quasi-copulas but until then we do not assume even that they are 1-increasing. As before, we assume that they are defined on $\DD$ which is either $\II^2$ in the general case or a mesh $\Delta$ in the discrete case. For a rectangle $R$ with {distinct} standard vertices $\mathbf{a}, \mathbf{b}, \mathbf{c}, \mathbf{d}$ we define the \textsl{main corner set} $M(R)=\{\mathbf{a},\mathbf{c}\}$ and the \textsl{opposite corner set} $O(R)= \{\mathbf{b},\mathbf{d}\}$. Given a rectangle $R$ we define for any point $\mathbf{x}\in\II^2$ its multiplicity by
\[
    m_R(\mathbf{x})=\left\{
           \begin{array}{ll}
             1, & \hbox{if $\mathbf{x}\in M(R)$;} \\
             -1, & \hbox{if $\mathbf{x}\in O(R)$;} \\
             0, & \hbox{otherwise.}
           \end{array}
         \right.
\]
Let us extend this definition to any $R\in\mathfrak{R}$, the set of all disjoint unions of rectangles, i.e.\ if $\{R_i\}_{i=1}^{n}$ is an arbitrary finite set of rectangles, then an element of $\mathfrak{R}$ is of the form $R=\bigsqcup_{i=1}^nR_i$, where $\bigsqcup$ denotes the disjoint union, and we let $m_R(\mathbf{x})= \sum_{i=1}^n m_{R_i}(\mathbf{x})$. We are now in position to give the definition for the volume of an element $R\in \mathfrak{R}$ corresponding to the real valued function $A$
\[
    V_A(R)=\sum_{\mathbf{x}\in\II^2}A(\mathbf{x})m_R(\mathbf{x})
\]
and in the same way for $B$. It is obvious that this sum is actually finite. It is also clear that when specializing to quasi-copulas and rectangles this definition coincides with the usual definition of the volume.

\begin{lemma}\label{lemma:additive}
  The multiplicity of a given point and the volume corresponding to a given real valued function $A$ are additive:
\begin{description}
  \item[(a)] $m_{R_1\sqcup R_2}(\mathbf{x})=m_{R_1}(\mathbf{x})+m_{ R_2}(\mathbf{x})$;
  \item[(b)] $V_A(R_1\sqcup R_2)=V_A(R_1)+V_A( R_2)$.
\end{description}
\end{lemma}

Let us define a function $L$ of $R\in\mathfrak{R}$, and functions $P_M$ and $P_O$ of $\mathbf{x}\in\II^2$, all depending also on the real valued functions $A$ and $B$
\begin{equation}\label{eq:LPmPo}
\begin{split}
   L^{(A,B)}(R) & =\sum_{\substack{\mathbf{y}\in\II^2\\m_R(\mathbf{y})>0}} B(\mathbf{y})m_R(\mathbf{y}) + \sum_{\substack{\mathbf{y}\in\II^2\\ m_R(\mathbf{y})<0}}A(\mathbf{y})m_R(\mathbf{y}) \\
   P_M^{(A,B)}(\mathbf{x})  & =\inf_{\substack{R\in\mathfrak{R}\\ m_R(\mathbf{x})>0}}\frac{L^{(A,B)}(R)}{m_R(\mathbf{x})} \quad\mbox{and}\quad P_O^{(A,B)}(\mathbf{x})= \inf_{\substack{R\in \mathfrak{R}\\ m_R(\mathbf{x})<0}} \frac{L^{(A,B)}(R)}{-m_R(\mathbf{x})},
\end{split}
\end{equation}
where infimum of an empty set is assumed equal to $+\infty$.

In the following propositions we will assume two conditions on the pair of real valued functions $(A,B)$ (to be later specialized to an imprecise copula):
\begin{description}
  \item[(Q1)] $A\leqslant B$, and
  \item[(Q2)] $L^{(A,B)}(R)\geqslant0$ for all $R\in\mathfrak{R}$.
\end{description}

In the following proposition we need the function $\gamma^{(A,B)} (\mathbf{x})= \min\{P_O^{(A,B)}(\mathbf{x}), B(\mathbf{x})- A(\mathbf{x})\}$ defined for $\mathbf{x}\in\DD$.

\begin{proposition}\label{prop:main}
  Let the pair of real valued functions $(A,B)$ satisfy Conditions $\mathbf{(Q1)}$, $\mathbf{(Q2)}$ and let there exist an $\mathbf{x} \in\DD$ such that $t_0=\gamma^{(A,B)}(\mathbf{x})>0$. Then the pair of real valued functions $(A',B)$, where
\[
    A'(\mathbf{y})=\left\{
                     \begin{array}{ll}
                       A(\mathbf{x})+t, & \hbox{if $\mathbf{y}= \mathbf{x}$;} \\
                       A(\mathbf{y}), & \hbox{otherwise;}
                     \end{array}
                   \right.
\]
satisfies conditions $\mathbf{(Q1)}, \mathbf{(Q2)}$ for any $t, 0<t\leqslant  t_0$. If we choose $t=t_0$, then $\gamma^{(A',B)} (\mathbf{x})=0$.
\end{proposition}

\begin{proof}
  Clearly, we only need to show \textbf{(Q2)}. Choose any $R\in \mathfrak{R}$ and assume first that $m_R(\mathbf{x})\geqslant0$. Then $L^{(A',B)}(R)=L^{(A,B)}(R)\geqslant0$ because functions $A$ and $A'$ differ only at the chosen point $\mathbf{x}$ which appears in $L$ only as an argument of $B$ by \eqref{eq:LPmPo}. Now, if $m_R(\mathbf{x})<0$, then
\[
    t\leqslant t_0\leqslant P_O^{(A,B)}(\mathbf{x})\leqslant\frac{L^{(A,B)}(R)}{-m_R(\mathbf{x})}
\]
so that
\[
    L^{(A',B)}(R)=L^{(A,B)}(R)+tm_R(\mathbf{x})\geqslant0.
\]
Choose $t=t_0$. If $t_0=B(\mathbf{x})-A(\mathbf{x})$, then $B(\mathbf{x})-A'(\mathbf{x})=0$ and we are done. If $t_0=P_O^{(A,B)}(\mathbf{x})$ then for every $\varepsilon>0$ there is an $R\in\mathfrak{R}$ such that $m_R(\mathbf{x})<0$ and that $L^{(A,B)}(R)\leqslant -m_R(\mathbf{x})(t_0+\varepsilon)$ so that $L^{(A',B)}(R)\leqslant -m_R(\mathbf{x})\varepsilon$. Therefore, $P_O^{(A',B)}(\mathbf{x})\leqslant \varepsilon$ for all $\varepsilon>0$ and we are done again.
\end{proof}

\textbf{Remark.} Note that even if we started with quasi-copulas $(A,B)$ the function $A'$ would not be a quasi-copula in general.

\begin{proposition}\label{plus}
  Under the conditions $\mathbf{(Q1)},\mathbf{(Q2)}$ we have
\[
    P_M^{(A,B)}(\mathbf{x})+ P_O^{(A,B)}(\mathbf{x}) \geqslant B(\mathbf{x})-A(\mathbf{x})
\]
for all $\mathbf{x}\in[0,1]^2$.
\end{proposition}

\begin{proof}
  Let $R_1,R_2\in\mathfrak{R}$ be such that $m_{R_1}(\mathbf{x})<0$ and $m_{R_2}(\mathbf{x})>0$. Let us observe that
\begin{equation}\label{eq:quotient1}
  \frac{L^{(A,B)}(R_1)}{-m_{R_1}(\mathbf{x})} + \frac{L^{(A,B)}(R_2)}{m_{R_2}(\mathbf{x})} -(B(\mathbf{x})-A(\mathbf{x}))  = \frac{ S(R_1,R_2) } {(-{m_{R_1}(\mathbf{x})}){m_{R_2}(\mathbf{x})}},
\end{equation}
where
\[
\begin{split}
   S(R_1,R_2) & ={m_{R_2}(\mathbf{x})} L^{(A,B)}(R_1)+ (-{m_{R_1}(\mathbf{x})}) L^{(A,B)}(R_2) \\
     & -(-{m_{R_1}(\mathbf{x})}){m_{R_2}(\mathbf{x})} (B(\mathbf{x})-A(\mathbf{x})).
\end{split}
\]
After introducing
\[
    R_3=\left(\bigsqcup_{i=1}^{m_{R_2}(\mathbf{x})}R_1\right) \sqcup \left(\bigsqcup_{j=1}^{-m_{R_1}(\mathbf{x})}R_2\right)
\]
we want to show that
\begin{equation}\label{S}
  S(R_1,R_2)\geqslant L^{(A,B)}(R_3)
\end{equation}
the right-hand side of which is no smaller than zero by \textbf{(Q2)}; the proposition will then follow after taking the infima of the two quotients on the left-hand side of \eqref{eq:quotient1}. Using the fact that multiplicity of corners is additive by Lemma \ref{lemma:additive}\textbf{(a)} we observe that $m_{R_3}(\mathbf{y})=m_{R_2}(\mathbf{x})m_{R_1}(\mathbf{y})+ (-m_{R_1}(\mathbf{x}))m_{R_2}(\mathbf{y})$ for any point $\mathbf{y}$.
Recall the definition \eqref{eq:LPmPo} of functions $L$  to see that the first two terms of the function $S$ can be expanded as
\begin{align*}
&\sum_{m_{R_1}(y)>0} m_{R_2}(x)B(y)m_{R_1}(y) +\sum_{m_{R_1}(y)<0} m_{R_2}(x)A(y)m_{R_1}(y) +\\
&\sum_{m_{R_2}(y)>0} (-m_{R_1}(x))B(y)m_{R_2}(y) +\sum_{m_{R_2}(y)<0} (-m_{R_1}(x))A(y)m_{R_2}(y).
\end{align*}
With respect to the point $\mathbf{y}\neq\mathbf{x}$ in these sums we will consider four cases. Assume at first ${m_{R_1}(\mathbf{y})>0}$ and ${m_{R_2}(\mathbf{y})>0}$. Then the contribution of $\mathbf{y}$ to $S$ equals
\[
    B(\mathbf{y})(m_{R_2}(\mathbf{x})m_{R_1}(\mathbf{y})+ (-m_{R_1}(\mathbf{x}))m_{R_2}(\mathbf{y}))
\]
which is exactly equal to the contribution of $\mathbf{y}$ to the right-hand side of \eqref{S}. Next, assume  ${m_{R_1}(\mathbf{y})>0}$ and ${m_{R_2}(\mathbf{y})\leqslant0}$ to get that the contribution of  $\mathbf{y}$ to $S$ equals
\[
    B(\mathbf{y})m_{R_2}(\mathbf{x})m_{R_1}(\mathbf{y})+ A(\mathbf{y})(-m_{R_1}(\mathbf{x}))m_{R_2}(\mathbf{y}).
\]
In this case we are making this expression not greater when we replace either $B(\mathbf{y})$ with $A(\mathbf{y})$ or $A(\mathbf{y})$ with $B(\mathbf{y})$. So, this is not smaller than the contribution of $\mathbf{y}$ to the right-hand side of \eqref{S}. The case that ${m_{R_1}(\mathbf{y})\leqslant0}$ and ${m_{R_2}(\mathbf{y})>0}$ goes similarly. Now, if ${m_{R_1}(\mathbf{y})\leqslant0}$ and ${m_{R_2}(\mathbf{y})\leqslant0}$. Then the contribution of $\mathbf{y}$ to $S$ equals
\[
    A(\mathbf{y})(m_{R_2}(\mathbf{x})m_{R_1}(\mathbf{y})+ (-m_{R_1}(\mathbf{x}))m_{R_2}(\mathbf{y}))
\]
which is exactly equal to the contribution of $\mathbf{y}$ to the right-hand side of \eqref{S}. Finally, if $\mathbf{y}=\mathbf{x}$, then we get zero contribution on both sides of the desired inequality which finishes the proof.
\end{proof}

\begin{theorem}\label{thm:main}
  If under the conditions $\mathbf{(Q1)},\mathbf{(Q2)}$ we have
\begin{equation}\label{eq:copula}
  \min\{P_O^{(A,B)}(\mathbf{x}),B(\mathbf{x})-A(\mathbf{x})\}=0\quad \mbox{for all}\quad\mathbf{x},
\end{equation}
then $V_A(R)\geqslant0$ for all rectangles $R$.
\end{theorem}

\begin{proof}
  We will prove this by contradiction. So, assume that $\mathbf{(Q1)}, \mathbf{(Q2)}$, and \eqref{eq:copula} hold and that there exists a rectangle $R$ such that $V_A(R)=v<0$. Let $\mathbf{x}_1$ respectively $\mathbf{x}_2$ be the southwest respectively the northeast corner of $R$. %We first want to show that $P_O^{(A,B)}$ equals zero at these two points.
Since $\min\{P_O^{(A,B)}(\mathbf{x}_1),B(\mathbf{x}_1) -A(\mathbf{x}_1)\}=0$ we have either $P_O^{(A,B)}(\mathbf{x}_1)=0$ or $B(\mathbf{x}_1)-A(\mathbf{x}_1)=0$. If $P_O^{(A,B)}(\mathbf{x}_1)=0$ we have by Proposition \ref{plus} that
\[
    P_M^{(A,B)}(\mathbf{x}_1)\geqslant B(\mathbf{x}_1)-A(\mathbf{x}_1).
\]
If $B(\mathbf{x}_1)-A(\mathbf{x}_1)=0$ then we get the same conclusion by \textbf{(Q1)}. Since $m_R(\mathbf{x}_1)=1$ it follows immediately that
\[
    L^{(A,B)}(R)\geqslant P_M^{(A,B)}(\mathbf{x}_1)\geqslant B(\mathbf{x}_1)-A(\mathbf{x}_1).
\]
By the definition of $L$ this inequality is equivalent to
\[
    B(\mathbf{x}_2)-A(\mathbf{x}_2)\geqslant -v>0.
\]
Assumption \eqref{eq:copula} implies that $P_O^{(A,B)}(\mathbf{x}_2) =0$ and similar arguments yield $P_O^{(A,B)}(\mathbf{x}_1) =0$. By the definition of this function there exist elements $R_1,R_2\in \mathfrak{R}$ such that $m_{R_1}(\mathbf{x}_1)<0$, $m_{R_2}(\mathbf{x}_2)<0$, and such that
\begin{equation}\label{eq:estimates}
  \frac{L^{(A,B)}(R_1)}{-m_{R_1}(\mathbf{x}_1)}<-\frac{v}{2} \quad\mbox{and}\quad \frac{L^{(A,B)}(R_2)} {-m_{R_2}(\mathbf{x}_2)} <-\frac{v}{2}.
\end{equation}
We want to show that $m_{R_2}(\mathbf{x}_1) \leqslant 0$. We will prove this by a contradiction. Assume that $m_{R_2}(\mathbf{x}_1) >0$ and introduce
\[
R_3=R_2\sqcup \left (\bigsqcup_{i=1}^{-m_{R_2}(\mathbf{x}_2)}R\right)
\]
so that $m_{R_3}(\mathbf{x}_1)=m_{R_2}(\mathbf{x}_1)+ (-m_{R_2}(\mathbf{x}_2))>0$ and consequently
\begin{equation}\label{eq:positive1}
  P_M^{(A,B)}(x_1)\leqslant \frac{L^{(A,B)}(R_3)} {m_{R_3}(\mathbf{x}_1)}.
\end{equation}
Recall the considerations of the proof of Proposition \ref{plus} and estimate contributions of points $\mathbf{y}$ similarly to get
\[
  L^{(A,B)}(R_3)\leqslant L^{(A,B)}(R_2)+(B(\mathbf{x}_1)- A(\mathbf{x}_1)+ V_A(R))(-m_{R_2}(\mathbf{x}_2)),
\]
where we use the right one of the estimates \eqref{eq:estimates} and the fact that $V_A(R)=v$ to get
\begin{equation}\label{eq:positive2}
  L^{(A,B)}(R_3)< \left(\frac{v}{2}+B(\mathbf{x}_1)- A(\mathbf{x}_1)\right) (-m_{R_2}(\mathbf{x}_2)).
\end{equation}
Using Proposition \ref{plus} and combining inequalities \eqref{eq:positive1} and \eqref{eq:positive2} we get
\begin{equation*}
  \begin{split}
   0=P_O^{(A,B)}(x_1) & \geqslant -P_M^{(A,B)}(x_1)+B(\mathbf{x}_1)- A(\mathbf{x}_1) \\
     & >\frac{m_{R_2}(\mathbf{x}_2)}{m_{R_3}(\mathbf{x}_1)}\left(\frac{v}{2}+B(\mathbf{x}_1)- A(\mathbf{x}_1)\right)+B(\mathbf{x}_1)- A(\mathbf{x}_1)\\ &=\frac{m_{R_2}(\mathbf{x}_2)}{m_{R_3}(\mathbf{x}_1)}\,\frac{v}{2}+ \frac{m_{R_2} (\mathbf{x}_2)+m_{R_3}(\mathbf{x}_1) }{m_{R_3}(\mathbf{x}_1)} (B(\mathbf{x}_1)- A(\mathbf{x}_1)) \geqslant0.
\end{split}
\end{equation*}
This contradiction proves that $m_{R_2}(\mathbf{x}_1) \leqslant 0$. Similarly we get $m_{R_1}(\mathbf{x}_2) \leqslant 0$.\\

On the final step of the proof we introduce
\[
    R_4=\left(\bigsqcup_{i=1}^{-m_{R_2}(\mathbf{x}_2)}R_1\right) \sqcup \left(\bigsqcup_{j=1}^{-m_{R_1}(\mathbf{x}_1)}R_2\right) \sqcup\left(\bigsqcup_{k=1}^{(-m_{R_1}(\mathbf{x}_1)) (-m_{R_2}(\mathbf{x}_2))}R\right).
\]
Note that $m_{R_4}(\mathbf{x}_1)=(-m_{R_1}(\mathbf{x}_1))m_{R_2} (\mathbf{x}_1)\leqslant0$ and $m_{R_4}(\mathbf{x}_2)= (-m_{R_2} (\mathbf{x}_2)) m_{R_1}(\mathbf{x}_2)\leqslant0$.
Using the usual considerations of the detailed contributions of points and estimates \eqref{eq:estimates} we get
\begin{equation*}
  \begin{split}
     L^{(A,B)}(R_4) & \leqslant (-m_{R_2}(\mathbf{x}_2))L^{(A,B)}(R_1) + (-m_{R_1}(\mathbf{x}_1))L^{(A,B)}(R_2) + m_{R_1}(\mathbf{x}_1) m_{R_2}(\mathbf{x}_2) V_A(R) \\
       & <  (-m_{R_2}(\mathbf{x}_2)) m_{R_1}(\mathbf{x}_1) \frac{v}{2} + (-m_{R_1}(\mathbf{x}_1))m_{R_2}(\mathbf{x}_2) \frac{v}{2}  + m_{R_1}(\mathbf{x}_1) m_{R_2}(\mathbf{x}_2) v = 0
  \end{split}
\end{equation*}
in contradiction with \textbf{(Q2)} thus proving the desired result.
\end{proof}

In the following proposition we fix a mesh $\Delta=\delta_x\times \delta_y$. In the set of unions of rectangles $\mathfrak{R}$ we will consider only those made of rectangles with vertices from the mesh.

\begin{proposition}\label{main_proposition}
  Let $A\leqslant B$ be discrete quasi-copulas. Then, there exists a discrete copula $C$ with $A\leqslant C\leqslant B$ if and only if
\[
    L^{(A,B)}(R) \geqslant0
\]
for all $R\in \mathfrak{R}$.
\end{proposition}

\begin{proof}
  We first assume that a copula $C$ like that exists. Then for every $R\in \mathfrak{R}$ we estimate
\[
L^{(A,B)}(R) \geqslant L^{(C,C)}(R) = V_C(R)\geqslant 0.
\]
So, our condition is necessary. Let us show that it is also sufficient. Assume that $L^{(A,B)}(R) \geqslant0$. This implies that the function $\gamma^{(A,B)}(\mathbf{x}) =\min\{P_O^{(A,B)} (\mathbf{x}), B(\mathbf{x})-A(\mathbf{x})\}$ is nonnegative for all $\mathbf{x}$ in the mesh.

Note that any $C$ such that $A\leqslant C\leqslant B$ is automatically grounded and has 1 as a neutral element since $A$ and $B$ have these properties. So, we only have to show nonnegativity of the volumes of rectangles. Choose a point $\mathbf{x}_0\in\Delta$ such that $t= \gamma^{(A,B)}(\mathbf{x}_0) >0$. If no such point exists, we have reached the desired conclusion by Theorem \ref{thm:main}. By Proposition \ref{prop:main} we can replace function $A$ by function $A'$ defined by $A'(\mathbf{x}_0)=A(\mathbf{x}_0)+t$ and $A'(\mathbf{x})=A(\mathbf{x})$ for all $\mathbf{x}\neq \mathbf{x}_0$ in the mesh. It follows clearly that $B \geqslant A' \geqslant A$ and therefore $L^{(A,B)}(R) \geqslant L^{(A',B)}(R)$, $P_O^{(A,B)} (\mathbf{x}) \geqslant P_O^{(A',B)}(\mathbf{x})$ and $\gamma^{(A,B)} (\mathbf{x}) \geqslant \gamma^{(A',B)}(\mathbf{x})$ for all $R$ and $\mathbf{x}$. Proposition \ref{prop:main} implies $L^{(A',B)}(R) \geqslant 0$ for all $R$, hence $P_O^{(A',B)}(\mathbf{x}) \geqslant 0$ and $\gamma^{(A',B)}(\mathbf{x}) \geqslant 0$ for all $\mathbf{x}$. In addition, $\gamma^{(A',B)}(\mathbf{x}_0)=0$ again by Proposition \ref{prop:main}. We can repeat this procedure for any $\mathbf{x}_0$ such that $\gamma^{(A,B)}(\mathbf{x}_0)>0$. Since the mesh is finite we are done in a finite number of steps.
\end{proof}

\begin{theorem}\label{main_theorem}
  Let $A\leqslant B$ be quasi-copulas. Then, there exists a copula $C$ with $A\leqslant C\leqslant B$ if and only if
\[
    L^{(A,B)}(R) \geqslant0
\]
for all $R\in \mathfrak{R}$.
\end{theorem}

\begin{proof}
  We first assume that there exists a copula $C$ such that $A\leqslant C\leqslant B$ and choose an $R\in \mathfrak{R}$. Note that $R$ is made of a finite number of rectangles that have a finite union of all possible corners. So, there exists a mesh $\Delta$ containing all these corners. Now, observe that $\langle A\rangle=A|_\Delta$ and $\langle B\rangle=B|_\Delta$ are quasi-copulas, that $\langle C\rangle= C|_\Delta$ is a copula by the remark immediately following the statement of Proposition \ref{copula_properties}, and that $\langle A\rangle\leqslant \langle C\rangle\leqslant\langle B\rangle$. So, the desired conclusion follows by Proposition \ref{main_proposition}.

To get the inverse implication, assume that the condition of the theorem is fulfilled for all $R\in\mathfrak{R}$. Choose a sequence of meshes
$\Delta_n\subseteq\Delta_{n+1}$ for $n\in\mathds{N}$ whose union of corners is dense in $\II^2$. (One may choose, say, $\Delta_n= \delta_{x^n} \times \delta_{y^n}$ determined by points
\[
    \delta_{x^n}=\delta_{y^n}=\left\{\frac{k}{2^n}\right\}_{k=0}^{2^n}.
\]
for $n=1,2,\ldots$.) Now, fix an $n\in\mathds{N}$, let $\mathfrak{R}_n$ be the set of disjoint unions of rectangles with corners in $\Delta_n$, and let $\langle A\rangle_n=A|_{\Delta_n}$ and $\langle B\rangle_n =B|_{\Delta_n}$. Then, these objects satisfy the assumptions of Proposition \ref{main_proposition}, so that there exists a discrete copula $C_n$ on $\Delta_n$ such that $\langle A\rangle_n\leqslant C_n\leqslant \langle B\rangle_n$. For any $n\in\mathds{N}$ extend the discrete copula $C_n$ to a general copula $\breve{C}_n=(C_n)^{BL}$. Since the set of copulas is compact by \cite[Theorem 1.7.7]{DuSe} there exists a subsequence $\breve{C}_{n_k},k\in \mathds{N}$, converging uniformly to a copula $C$. Now,
\[
    \breve{A}_n=\langle A\rangle_n^{\mathrm{BL}}\leqslant \breve{C}_n\leqslant \langle B\rangle_n^{\mathrm{BL}}=\breve{B}_n
\]
on $\DD=\II^2$ and by going twice to subsequences, if necessary, we may assume with no loss that $\breve{A}_{n_k}$ respectively $\breve{B}_{n_k}$ also converge to, say,  $\breve{A}$ respectively $\breve{B}$, so that, necessarily,
\[
    \breve{A}\leqslant C\leqslant \breve{B}.
\]
It remains to show that $\breve{A}=A$ and $\breve{B}=B$ in order to finish the proof of the theorem. On the way to see that we fix a $k\in \mathds{N}$ and choose a point $\mathbf{x}\in\Delta_{n_k}$. Clearly,
\[
    \breve{A}(\mathbf{x})=\breve{A}_{n_k}(\mathbf{x})= A|_{\Delta_{n_k}}(\mathbf{x})=A(\mathbf{x})\ \ \mbox{and similarly}\ \ \breve{B}(\mathbf{x})=B(\mathbf{x}).
\]
Since the union of all points of the kind is dense in $\DD=\II^2$ the desired result follows by the fact that quasi-copulas have the 1-Lipschitz property. Indeed, for any $\varepsilon>0$ we may choose $k\in\mathds{N}$ large enough such that all distances of the points in either $\delta_x$ or $\delta_y$ are no greater than $\displaystyle \frac{\varepsilon}{4}$. For any $\mathbf{x}\in\II^2$ we may then choose its closest point $\mathbf{y}$ in $\Delta_{n_k}$ and estimate
\[
    |\breve{A}(\mathbf{x})-A(\mathbf{x})|\leqslant |\breve{A}(\mathbf{x})- \breve{A}(\mathbf{y})| + |\breve{A}(\mathbf{y})-A(\mathbf{y})| + |A(\mathbf{y})-A(\mathbf{x})| \leqslant \frac{\varepsilon}{4} + \frac{\varepsilon}{4} +0+\frac{\varepsilon}{4}+\frac{\varepsilon}{4} = \varepsilon.
\]
Now, the leftmost expression of this display is independent of $k$, while the rightmost one can be made arbitrary small with $k$ going to infinity. This implies $\breve{A}(\mathbf{x})=A(\mathbf{x})$ and similarly for $B$.
\end{proof}

\section{The main results}\label{sec:main}

Using the theory developed so far we are now in position to give an example of an imprecise copula $(A,B)$ such that there is no copula $C$ between $A$ and $B$. Let us formulate the problem presented at the very end of \cite{DiSaPlMeKl} (cf.\ also \cite{PeViMoMi2,MoMiPeVi}) precisely. Define
\[
    \mathcal{C}(A,B)=\{C\,|\,C\ \mbox{copula}, A\leqslant C\leqslant B\} = [A,B]\cup\mathcal{C},
\]
where $\mathcal{C}$ is the set of all copulas and the ordered interval $[A,B]$ was defined in Section \ref{sec:bounds} immediately after Lemma \ref{lem:interval}. Observe that in this definition we only need $A$ and $B$ to be functions on $\II^2$ such that $A\leqslant B$ which is fulfilled, in particular, whenever $(A,B)$ is an imprecise copula. Let us propose the question in three steps:
\begin{enumerate}[Q.I]
  \item When is $\mathcal{C}(A,B)$ nonempty? In particular, can it be empty for some imprecise copula $(A,B)$?
  \item If $\mathcal{C}(A,B)\neq\emptyset$ is it possible that
\[
    \bigvee\mathcal{C}(A,B)=B ?
\]
  \item If $\mathcal{C}(A,B)\neq\emptyset$ is it possible that
\[
    \bigwedge\mathcal{C}(A,B)=A ?
\]
\end{enumerate}

Theorem \ref{main_theorem} gives a quite general answer to Q.I, saying that a pair of quasi-copulas $A,B$ with $A\leqslant B$ has $\mathcal{C}(A,B)$ nonempty if an only if $L^{(A,B)}(R) \geqslant0$. To answer the proposed question it suffices to solve the discrete version of the problem. Indeed, if we find a discrete imprecise copula $(\langle A\rangle,\langle B\rangle)$ such that there is no discrete copula $\langle C\rangle$ with the property $\langle A\rangle\leqslant\langle C\rangle\leqslant\langle B\rangle$, then we know by Proposition \ref{main_proposition} that there exists a disjoint union of rectangles $R$ (with corners in the corresponding mesh) such that $L^{(\langle A\rangle,\langle B\rangle)}(R)\ngeqslant0$. It now suffices to extend this discrete imprecise copula to a general one, say, by defining
\[
    A=\langle A\rangle^{\mathrm{BL}}\ \ \mbox{and}\ \ B=\langle B\rangle^{\mathrm{BL}}
\]
to get
\begin{equation}\label{anticondition}
  L^{(A,B)}(R)=L^{(\langle A\rangle,\langle B\rangle)}(R)\ngeqslant0
\end{equation}
and the imprecise copula $(A,B)$ has the desired property by Theorem \ref{main_theorem}.

Towards the discrete example we fix a mesh $\Delta=\delta_x\times \delta_y$ where $\displaystyle x_k=y_k =\frac{k}{10}$ for $k=0,1,\ldots, 10$, and let the value of a discrete imprecise copula $(\langle A\rangle, \langle B\rangle)$ be given firstly by
\[
    \langle A\rangle\left(\frac{j}{10},\frac{k}{10}\right)=[A]_{j+1,k+1}\ \ %\mbox{and}\ \ \langle B\rangle \left(\frac{j}{7},\frac{k}{7}\right)= [B]_{j+1,k+1}\ \
\mbox{for}\ \ j,k=0,1,\ldots,10,
\]
where $[A]$ is the matrix
\[
A=\frac{1}{50}\left[
\begin{array}{ccccccccccc}
 0 & 0 & 0 & 0 & 0 & 0 & 0 & 0 & 0 & 0 & 0 \\
 0 & 0 & 1 & 2 & 3 & 4 & 5 & 5 & 5 & 5 & 5 \\
 0 & 1 & 2 & 3 & 3 & 4 & 5 & 10 & 10 & 10 & 10 \\
 0 & 2 & 2 & 5 & 7 & 7 & 8 & 13 & 15 & 15 & 15 \\
 0 & 3 & 3 & 6 & 7 & 7 & 8 & 13 & 18 & 20 & 20 \\
 0 & 4 & 4 & 6 & 9 & 11 & 11 & 16 & 21 & 25 & 25 \\
 0 & 5 & 5 & 7 & 9 & 11 & 11 & 16 & 21 & 26 & 30 \\
 0 & 5 & 10 & 12 & 14 & 16 & 16 & 21 & 26 & 31 & 35 \\
 0 & 5 & 10 & 15 & 19 & 21 & 21 & 26 & 31 & 36 & 40 \\
 0 & 5 & 10 & 15 & 20 & 25 & 26 & 31 & 36 & 41 & 45 \\
 0 & 5 & 10 & 15 & 20 & 25 & 30 & 35 & 40 & 45 & 50 \\
\end{array}
\right],
\]
and secondly, by defining $\langle B\rangle=\langle A\rangle_M$. A short computation reveals that the matrix corresponding to $D_M^{\langle A\rangle}$ equals
\[
\left[D_M^{\langle A\rangle}\right] = \frac{1}{50}\left[
\begin{array}{ccccccccccc}
 0 & 0 & 0 & 0 & 0 & 0 & 0 & 0 & 0 & 0 & 0 \\
 0 & -1 & -1 & -1 & -1 & -1 & 0 & 0 & 0 & 0 & 0 \\
 0 & -1 & 0 & 0 & -1 & -1 & -1 & 0 & 0 & 0 & 0 \\
 0 & -1 & -1 & -1 & 0 & -1 & -1 & -1 & 0 & 0 & 0 \\
 0 & -1 & -1 & 0 & -1 & -1 & -1 & -1 & 0 & 0 & 0 \\
 0 & -1 & -1 & -1 & 0 & 0 & -1 & -1 & -1 & 0 & 0 \\
 0 & 0 & -1 & -1 & -1 & -1 & -1 & -1 & -1 & 0 & 0 \\
 0 & 0 & 0 & -1 & -1 & -1 & -1 & -1 & -1 & 0 & 0 \\
 0 & 0 & 0 & 0 & 0 & 0 & -1 & -1 & -1 & 0 & 0 \\
 0 & 0 & 0 & 0 & 0 & 0 & 0 & 0 & 0 & 0 & 0 \\
 0 & 0 & 0 & 0 & 0 & 0 & 0 & 0 & 0 & 0 & 0 \\
\end{array}
\right],
\]
so that
\[
    \langle B\rangle\left(\frac{j}{10},\frac{k}{10}\right)=[B]_{j+1,k+1}\ \ \mbox{for}\ \ j,k=0,1,\ldots,10,
\]
where $[B]$ is the matrix
\[
    [B]=[A]- \left[D_M^{\langle A\rangle}\right] =\frac{1}{50}\left[
\begin{array}{ccccccccccc}
 0 & 0 & 0 & 0 & 0 & 0 & 0 & 0 & 0 & 0 & 0 \\
 0 & 1 & 2 & 3 & 4 & 5 & 5 & 5 & 5 & 5 & 5 \\
 0 & 2 & 2 & 3 & 4 & 5 & 6 & 10 & 10 & 10 & 10 \\
 0 & 3 & 3 & 6 & 7 & 8 & 9 & 14 & 15 & 15 & 15 \\
 0 & 4 & 4 & 6 & 8 & 8 & 9 & 14 & 18 & 20 & 20 \\
 0 & 5 & 5 & 7 & 9 & 11 & 12 & 17 & 22 & 25 & 25 \\
 0 & 5 & 6 & 8 & 10 & 12 & 12 & 17 & 22 & 26 & 30 \\
 0 & 5 & 10 & 13 & 15 & 17 & 17 & 22 & 27 & 31 & 35 \\
 0 & 5 & 10 & 15 & 19 & 21 & 22 & 27 & 32 & 36 & 40 \\
 0 & 5 & 10 & 15 & 20 & 25 & 26 & 31 & 36 & 41 & 45 \\
 0 & 5 & 10 & 15 & 20 & 25 & 30 & 35 & 40 & 45 & 50 \\
\end{array}
\right].
\]
The pair $(\langle A\rangle,\langle B\rangle)$ is a discrete imprecise copula by Theorem \ref{interval}. It remains to find a disjoint union $R$ of rectangles with corners in this mesh such that Condition \eqref{anticondition} is fulfilled. In order to find the right $R$ we write down the volumes with respect to $A$ of the small rectangles determining the mesh:
\newcommand{\cc}{\cellcolor{black!15!white}}
$$V=\frac{1}{50}\left[
\begin{array}{cccccccccc}
 0 & 1 & 1 & 1 & 1 & 1 & 0 & 0 & 0 & 0 \\
 1 & \cc 0 & \cc 0 & \cc -1 & \cc 0 & \cc 0 & 5 & 0 & 0 & 0 \\
 1 & \cc -1 & 2 & 2 & \cc -1 & \cc 0 & 0 & 2 & 0 & 0 \\
 1 & \cc 0 & \cc 0 & \cc -1 & \cc 0 & \cc 0 & 0 & 3 & 2 & 0 \\
 1 & \cc 0 & \cc -1 & 2 & 2 & \cc -1 & 0 & 0 & 2 & 0 \\
 1 & \cc 0 & \cc 0 & \cc -1 & \cc 0 & \cc 0 & 0 & 0 & 1 & 4 \\
 0 & 5 & 0 & 0 & 0 & 0 & 0 & 0 & 0 & 0 \\
 0 & 0 & 3 & 2 & 0 & 0 & 0 & 0 & 0 & 0 \\
 0 & 0 & 0 & 1 & 3 & 1 & 0 & 0 & 0 & 0 \\
 0 & 0 & 0 & 0 & 0 & 4 & 0 & 0 & 0 & 1 \\
\end{array}
\right]$$
We can think of the entries of this matrix as a representation of the discrete quasi-copula in question on the given mesh. However, due to its matrix presentation the position of the main and the opposite corners seems to be interchanged. For example, the most northwest corner of the shaded region in the above image of matrix $V$ is a main one.

We observe two ``hills'' of hight $2/50$, presented as two holes in the shaded region, surrounded by seven depressions of depth $-1/50$, which are all contained in the shaded region. The rest of the volumes were chosen in such a way that they allowed us to construct the above pair of quasi-copulas. The $R$ that satisfies Condition \eqref{anticondition} is represented by the shaded region. So, it is defined exactly as the union of 21 small squares (i.e.\ of those that are determining the mesh) of the region and the desired fact now follows by a simple calculation. Observe that $V_{\langle A \rangle}(R)=-7$ and $R$ has $6$ corners with positive multiplicity, all with multiplicity $1$, and the value of $D_M^{\langle A \rangle}$ at all these corners is $-1$. Hence
\begin{align*}
L^{(\langle A \rangle,\langle B \rangle)}(R) &=\sum_{m_R(\mathbf{y}) \neq 0} \langle A \rangle(\mathbf{y})m_R(\mathbf{y})+ \sum_{m_R(\mathbf{y})>0}(\langle B \rangle(\mathbf{y})- \langle A \rangle(\mathbf{y})) m_R(\mathbf{y})=\\
&=V_{\langle A \rangle}(R)+\sum_{m_R(\mathbf{y})>0}(-D_M^{\langle A \rangle}(\mathbf{y}))m_R(\mathbf{y})=-7+6=-1.
\end{align*}

\begin{example}\label{maincounterex}
  There exists an imprecise copula $(A,B)$ such that $\mathcal{C}(A,B)= \emptyset$.
\end{example}

We now give answers to Q.II and Q.III.

\begin{theorem}\label{mainforward}
  Let $A\leqslant B$ be quasi-copulas and $\mathcal{C}(A,B)\neq \emptyset$. Then\vskip3mm
  \begin{enumerate}[(a)]
  \item $\displaystyle
    B=\bigvee\mathcal{C}(A,B)\ \ \mbox{if and only if}\ \ B(\mathbf{x}) -A(\mathbf{x}) \leqslant P_O^{(A,B)}(\mathbf{x})$ for all $x\in\II^2$.\\
  \item $\displaystyle
    A=\bigwedge\mathcal{C}(A,B)\ \ \mbox{if and only if}\ \ B(\mathbf{x})- A(\mathbf{x}) \leqslant P_M^{(A,B)}(\textbf{x})$ for all $x\in\II^2$.
  \end{enumerate}\vskip3mm
\end{theorem}

\begin{proof}
  Let us start by the proof of \emph{(a)}. Recall that the condition of Theorem \ref{main_theorem} is fulfilled. Assume first that condition $B(\mathbf{x})-A(\mathbf{x}) \leqslant P_O^{(A,B)}(\mathbf{x})$ is satisfied at a certain point $\mathbf{x}\in\II^2$. Choose a mesh, say $\Delta_n$ containing this point, recall the notation  $\gamma^{(A,B)} (\mathbf{x})= \min\{P_O^{(A,B)}(\mathbf{x}), B(\mathbf{x})- A(\mathbf{x})\}$ of Proposition \ref{prop:main}, and observe that $\gamma^{(A,B)} (\mathbf{x})=B(\mathbf{x})- A(\mathbf{x})$ in our case. So, using this proposition we may replace $A$ by $A'$ such that $A'(\mathbf{x})= B(\textbf{x})$ and $A'=A$ at all other points of the mesh. As in (the main part of) the proof of Theorem \ref{main_theorem} we continue correcting the values of $A'$ at other points of the mesh until we find a discrete copula $C_n$ such that $A|_{\Delta_n}\leqslant C_n\leqslant B|_{\Delta_n}$ and at the same time $C_n(\mathbf{x})= B(\mathbf{x})$. Continue as in that proof by a sequence of meshes each contained in the next one whose union of corners is dense in $\II^2$ and by an according sequence of discrete copulas $C_n$ extended to a sequence of general copulas $(C_n)^{\mathrm{BL}}$. By going to a subsequence, if necessary, we may achieve a uniformly convergent sequence and a limit copula $C$ such that $A\leqslant C\leqslant B$ and at the same time, $C(\mathbf{x})= B(\mathbf{x})$, thus proving one direction of \emph{(a)}.

To get the proof of \emph{(a)} in the other direction assume that $B=\bigvee\mathcal{C}(A,B)$, choose $\mathbf{x} \in \II^2$, $\varepsilon>0$, and $C\in\mathcal{C}(A,B)$ such that
\[
    C(\mathbf{x})>B(\mathbf{x})-\varepsilon.
\]
It is clear that $P_O^{(A,B)}(\mathbf{x})\geqslant P_O^{(A,C)}(\mathbf{x})$. We want to show that
\begin{equation}\label{eq:last}
  P_O^{(A,C)}(\mathbf{x})\geqslant C(\mathbf{x})-A(\mathbf{x}).
\end{equation}
This will imply $P_O^{(A,B)}(\mathbf{x})>B(\mathbf{x})-A(\mathbf{x}) -\varepsilon$ and the desired conlusion will follow by the fact that $\varepsilon$ can be chosen arbitrarily small. In the proof of \eqref{eq:last} we first recall that
\[
    P_O^{(A,C)}(\mathbf{x})= \inf_{\substack{R\in \mathfrak{R}\\ m_R(\mathbf{x})<0}} \frac{L^{(A,C)}(R)}{-m_R(\mathbf{x})},
\]
where
\[
    L^{(A,C)}(R)=\sum_{m_R(\mathbf{y})>0}C(\mathbf{y})m_R(\mathbf{y}) +\sum_{m_R(\mathbf{y})<0}A(\mathbf{y})m_R(\mathbf{y}).
\]
We add to and subtract from these sums the sum of $C(\mathbf{y}) m_R(\mathbf{y})$ over $\mathbf{y}\in\II^2$ with $m_R(\mathbf{y})<0$ to get
\[
     L^{(A,C)}(R)= V_C(R) + \sum_{m_R(\mathbf{y})<0} (C(\mathbf{y}) - A(\mathbf{y})) (-m_R(\mathbf{y})) \geqslant (C(\mathbf{x}) - A(\mathbf{x})) (-m_R(\mathbf{x}))
\]
because all the summands of the above sum are nonnegative and they also contain the summand with $\mathbf{y}=\mathbf{x}$. This implies Equation \eqref{eq:last} thus finishing the proof of \emph{(a)}.

The proof of \emph{(b)} follows by taking the reflection on the case \emph{(a)} \emph{mutatis mutandis}, i.e.\ once the necessary changes have been made. In particular, applying, say, $\sigma:(x,y) \mapsto(1-x,y)$ on Equations \eqref{eq:LPmPo} and noting that a reflection is exchanging the order on the lattice of quasi-copulas and by Lemma \ref{reflection} also the main and opposite  role of the corners of rectangles we first get
\[
    L^{(B^\sigma,A^\sigma)}(\sigma(R))=L^{(A,B)}(R)
\]
and then
\[
    P_O^{(B^\sigma,A^\sigma)}(\sigma(\mathbf{x}))=P_M^{(A,B)}(\mathbf{x})  %=\inf_{\substack{R\in\mathfrak{R}\\ m_R(\mathbf{x})>0}}\frac{L^{(A,B)}(R)}{m_R(\mathbf{x})}
    \quad\mbox{and}\quad P_M^{(B^\sigma,A^\sigma)}(\sigma(\mathbf{x}))= P_O^{(A,B)}(\mathbf{x}). %=\inf_{\substack{R\in \mathfrak{R}\\ m_R(\mathbf{x})<0}} \frac{L^{(A,B)}(R)}{-m_R(\mathbf{x})}
\]
So, the reflected \emph{(a)} becomes
\[
    B^\sigma=\bigwedge\mathcal{C}(B^\sigma,A^\sigma)\ \ \mbox{\emph{if and only if}}\ \ A^\sigma(\sigma(\mathbf{x})) -B^\sigma(\sigma(\mathbf{x})) \leqslant P_M^{(B^\sigma,A^\sigma)}(\sigma(\mathbf{x}))\ \ \mbox{\emph{for all}}\ \  x\in\II^2
\]
which is exactly \emph{(b)}. So, we are done by the first part of the proof.
\end{proof}

%\phantom{x}
\hfil \\

\begin{center}
  {\textsc{Conclusion}}
\end{center}

Example \ref{maincounterex} gives an imprecise copula according to the definition in \cite{MoMiPeVi} such that there is no copula contained in the order interval generated by it. So, as we pointed out in the abstract and explained further in the introduction, it is questionable whether the definition of an imprecise copula proposed there is in accordance with the intentions of the initiators. Of course, their Theorem 2(a) is still valid. Our approach through the functions $L, P_M,$ and $P_O$ might be helpful in improving their technique since Theorem \ref{mainforward} suggests a possible upgrade of their definition. However, the problem may be deeper since the authors of \cite{MoMiPeVi} never had the second half of the Sklar's theorem in the imprecise setting and developing a full scale imprecise theorem of Sklar's type would be a possible goal to attack. Actually, we believe that this is achievable with some more work, namely, one would also need to reconsider the definition of a coherent bivariate $p$-box as introduced in \cite{PeViMoMi2}. \\

\textbf{Acknowledgement.} The authors are thankful to Professors Susanne Saminger-Platz, Radko Mesiar, and Erich Peter Klement, for pointing this problem to us as well as for some interesting discussions on the problem.

% ----------------------------------------------------------------

\bibliographystyle{amsplain}

\end{document}